\documentclass[12pt,letterpaper]{article}

\usepackage{amssymb,amsmath,amsthm,mathrsfs,array,graphicx}
\usepackage{float}
\usepackage{multirow}
\usepackage{natbib}
\usepackage{subfigure}
\usepackage{caption}
\usepackage{wrapfig}
\usepackage{color}
\usepackage{authblk}
\usepackage{fancyhdr}
\usepackage{lipsum}
\usepackage{amsmath}
\usepackage{mathtools}
\usepackage{amssymb,amsmath,amsthm,mathrsfs,array,graphicx}
\usepackage{float}
\usepackage{multirow}
\usepackage{natbib}
\usepackage{subfigure}
\usepackage{color}
\usepackage{authblk}
\usepackage{fancyhdr}
\usepackage{mathtools}
\usepackage{textcomp}
\usepackage{siunitx}
\usepackage{enumitem}
\usepackage{tikz} 
	\usetikzlibrary{shapes,arrows}
\usepackage{microtype} 
\usepackage{tcolorbox}
\definecolor{mycolor}{rgb}{0.122, 0.435, 0.698}
\usepackage[margin=1in]{geometry}
\usepackage{algorithmic}

\numberwithin{equation}{section}

\newtheorem{remark}{Remark}
\newtheorem{definition}{Definition}

\newtheorem{assumption}{Assumption}

\newtheorem{theorem}{Theorem}
\newtheorem{proposition}{Proposition}
\newtheorem{hypo}{Hypothesis}

\usepackage[small]{titlesec}

\usepackage{hyperref}
    \hypersetup{
         unicode=true,
         colorlinks=true,
         linkcolor=blue,
         filecolor=magenta,      
         urlcolor=cyan,
         breaklinks=true
}

\usepackage{todonotes}
\begin{document}

\title{Einstein's Brownian motion model for chemotactic system and traveling band}
\author{Rahnuma Islam}
\author{Akif Ibraguimov}
\affil{Department of Mathematics, Texas Tech University} 

\date{}

\maketitle


\section{Abstract}
\label{abs}
We study the movement of the living organism in a band form towards the presence of chemical substrate based on a system of partial differential evolution equations. We incorporate the Einstein's method of Brownian motion to deduce the chemotactic model exhibiting travelling band. It is the first time that Einstein method has been used to motivate equations describing mutual interaction of chemotactic system. In addition to considering chemotactic response and the random motion of organism, we also consider the formation of crowd by organism via interactions within or between the community. This crowd effect can also be seen as any organism travel or migrate in a herd or group in search of food. We have shown that in the presence of limited and unlimited substrate traveling bands are achievable and it has been explained accordingly.


\section{Introduction}
\label{introduction}
The celebrated work of Einstein's theory of Brownian motion \cite{Einstein05} offered the existence of discrete molecule that are too small to be seen through a microscope but the resulting motion should be visible through microscope. In this theory, he argued that agitated particles in a suspended water are the results due to the collisions with molecules. Hence he constructed a model governing its motion with respect to nearby particles. Since then, the stochastic development of this approach has been incorporated into all the natural sciences, engineering, linguistics, finance, economics, and even the social sciences.\par

`Chemotaxis' is a biological phenomena by which organisms change their state of movements either toward or away from the chemical substance. This migration can be seen in cells ranging from bacteria to mammal. Cells of organism senses the higher gradient of chemoattractants and move in that direction. During this process of movement towards the chemical gradient, in a detailed inspection, the motion created by each individual cell appears to be erratic. This randomicity arises not only from the chemotactic response but also from the random jumps of cells. We argued that the Einstein's  theoretical framework of Brownian motion can describe the chemotactic response and random motion of organism. \par

In addition, we consider the formation of crowd by organism via interactions within or between the community. Many or most bacteria conduct cell-cell communication secreting chemical molecules, knows as, Quorum sensing \cite{quorum01}. This communication can happen both within and between bacterial species due to the presence of signal molecule, called "autoinducers" \cite{GOBBETTI200734} . In food related pathogen, after autoinducers make the bacteria aware of the existence of food, a certain threshold concentration of bacteria should be formed to trigger the event where independent bacteria accumulate into that formation and behave collectively.  This collective network bestows upon bacteria some advantages such as ability to migrate to a better environment containing more favourable resources or grow in a more cooperative fashion and increase the chance of survival and thriving. The suitability of our model, as per our expectation, is not just restricted to bacteria-sugar relation but also for any prey-predator interaction. For instance, a quorum response in vertebrate animal groups such as three-spine sticklebacks fish can be seen to play a role in the movement decisions of fish \cite{Ward6948}. Vertebrates use social cues and signal from a group and responds to the behavior if a certain threshold number of members is present in that group.\par

In this study, our assumptions involve two movements and interactions of organisms that takes in place simultaneously: interactions between and within organisms and movement of organisms towards substrate. Hence the formations of the crowd get affected by the response of organism to the presence of chemical substrate. Therefore, the distance between any two entity is proportional to the change in the distance between the entity and the food. Note that, the growth or reproduction is excluded from our model as traveling band is possible even in the absence of multiplicative cell. Also, chemical interactions between chemical substrates which form new components have not been considered.\par

In the Section~\ref{Einstein}, we will derive the chemotactic model motivated by Einstein's random walk model. We present exhibition of traveling band in two cases: environment with unlimited supply of food described in the Section~\ref{unlimited} and environment with limited supply of food explained in the Section~\ref{limited}. In the Section~\ref{dis}, some numerical results will be presented in support of our findings.


\section{Derivation of Einstein's model with consumption/reaction term}
\label{Einstein}

To formulate the partial differential equation (PDE) model, an existence of time interval $\tau$ between the collision of two particles is required. The interval $\tau$ is ``sufficiently small'' compared to the time scale $t$ of observation of the physical process, but not so small that the motions become correlated. Suppose $u(x,t)$ is the number of the particles (such as bacteria, glucose or predator, prey etc) per unit volume (density or concentration). Then, we will consider the following Einstein's general conservation law which gives the number of particles found at time $t+ \tau$ between two planes perpendicular to the $x$-axis, with abscissas $x$ and $x+ d x$, is given by

\begin{equation}\label{Eins_b_sys}
u(x, t+\tau) \cdot dx=  \left(\int_{-\infty}^{\infty} u(x+ \Delta, t) \varphi(\Delta) d \Delta +\int_{t}^{t+\tau} f(x,\xi)d\xi\right)\cdot dx.
\end{equation}

Time interval ($\tau$), distance traveled during the free jump ($\Delta$) and probability density function of jump ($\varphi$) can be functions of spatial distance $x$ and the time variable $t$ and of any other physical quantity such as density or number of particles etc. In our case, we will assume, for now, $\tau$ to be independent of concentration of particles $u$. And $\varphi(\Delta)$ is fixed with respect to $u(x,t)$. During the time interval $[t,t+\tau]$ in the unit volume around the particle located at the observation point $x$, it is possible that absorption and/or reaction with other particles (or with the suspending medium) occur. In Eq.~\eqref{Eins_b_sys}, $f(x,t)$ has been defined as the growth of the crowd of particles due to the chemotactic response per unit volume or the consumption rate by the particles per unit volume.\par 

Also we define the following basic properties:
\begin{definition}\label{axiom:Delta-e}(Expected value of the length of free jump)
\begin{equation*}
\Delta_{e}=\int \Delta\varphi(\Delta)d\Delta.
\end{equation*}
\end{definition}

\begin{definition}\label{sigma-b}(Standard variance of free jump)
\begin{equation*}
\sigma^2= \int(\Delta-\Delta_{e})^2 \varphi(\Delta) d \Delta.
\end{equation*}
\end{definition}

Now by Caratheodory theorem on differentiability  \cite{rah20}, there exists a function $\psi_{1}(x,t)$ such that for any smooth function $u(x,t)$
\begin{equation*}\label{Caratheodory}
u(x, t+ \tau)= u(x, t)+ \tau \psi_{1}(x,t+\tau) ,
\end{equation*}
where
\begin{equation*}
  \lim_{\tau\to 0} \psi_{1}(x,t+\tau)= \frac{\partial u(x,t)}{\partial t},
\end{equation*}
or
\begin{equation*}
 \psi_{1}(x,t+\tau)\approx \frac{\partial u(x,t)}{\partial t}.
\end{equation*}

similarly, for functions $\psi_{2}(x + \Delta ,t)$ and $\psi_{2}(x + \Delta ,t)$,
\begin{align*}
 \psi_{2}(x + \Delta_e ,t)&\approx \frac{\partial u(x,t)}{\partial x},\\
 \psi_{3}(x + \Delta_e ,t)&\approx \frac{\partial^2 u(x,t)}{\partial x^2}.
\end{align*}

And
\begin{equation*}\label{Caratheodory-x}
u(x+\Delta, t)= u(x, t)+ \Delta \psi_{2}(x+\Delta,t+\tau)\approx u(x,t)+ \Delta\frac{\partial u(x,t)}{\partial x}. 
\end{equation*}

Using above generic properties, we add and subtract $u(x+\Delta_e, t)$  on the right hand side of the Eq.~\eqref{Eins_b_sys} and then we compute as following
\begin{align}
&u(x, t+\tau)- u(x+ \Delta_e, t)\cdot dx = \nonumber\\
&\Bigg(\int_{-\infty}^{\infty} \bigg(u(x+ \Delta, t)-u(x+ \Delta_e,t)\bigg) \varphi(\Delta) d \Delta +\int_{t}^{t+\tau} f(x,\xi)d\xi \Bigg)\cdot dx. \nonumber 
\end{align}

After applying Charatheodory theorem for derivatives, we will get 
\begin{align}
&u(x,t)+\tau \psi_{1}(x,t+\tau)-u(x,t)-\Delta_{e} \psi_{2}(x+\Delta_{e},t) = \nonumber\\
&\int_{-\infty}^{\infty} \bigg(\psi_{2}(x+\Delta_{e},t) (\Delta-\Delta_e) +\psi_{3}(x+\Delta_e,t) (\Delta-\Delta_{e})^2\bigg)\varphi(\Delta) d \Delta \nonumber\\
& +\int_{t}^{t+\tau} f(x,\xi)d\xi .\nonumber
\end{align}

Using properties of the function $\psi$ for first and second derivatives in the vicinity of the point $(x,t)$, we will get
\begin{align}
&\tau \frac{\partial u}{\partial t}-\Delta_{e}\frac{\partial u}{\partial x} = \frac{\partial u}{\partial x}\int_{-\infty}^{\infty}( \Delta-\Delta_e) \varphi(\Delta)d \Delta+\frac{1}{2}\frac{\partial^2 u}{\partial x^2}\int_{-\infty}^{\infty}  (\Delta-\Delta_{e})^2 \varphi(\Delta) d \Delta  \nonumber\\
&+\int_{t}^{t+\tau} f(x,\xi)d\xi. \label{eins der}
\end{align}

With Definitions~\ref{axiom:Delta-e} and \ref{sigma-b}, Eq.~\eqref{eins der} becomes

\begin{equation}\label{final ein der}
    \tau \frac{\partial u}{\partial t}=\Delta_{e}\frac{\partial u}{\partial x}+\frac{1}{2} \sigma^2 \frac{\partial^2 u}{\partial x^2}+\int_{t}^{t+\tau} f(x,\xi)d\xi.
\end{equation}


\section{Derivation of chemotactic system when the availability of substrate is unlimited}\label{unlimited}
Let $u(x,t)$ and $v(x,t)$ be the concentration of organism and chemical substrate (food or any chemical attractor) per unit volume respectively with $x$ being the distance along the tube and $t$, the time.

\begin{figure}
    \centering
    \includegraphics[width=1\textwidth]{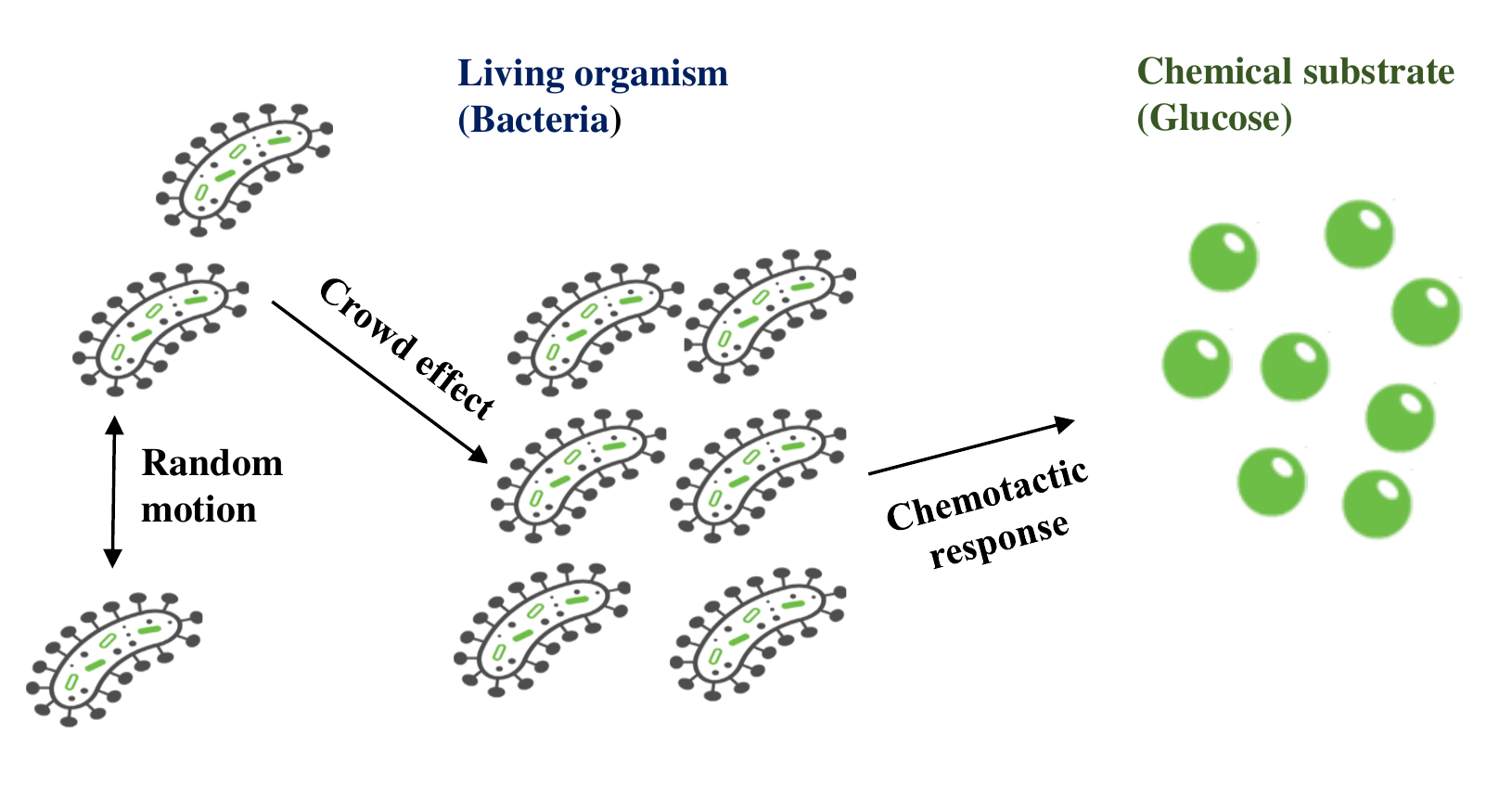}
    \caption{A virtual representation of the interactions between organism (bacteria) and chemical substrates (Glucose).}
    \label{fig:bacteria_food}
\end{figure}

The corresponding expression of the Eq.~\eqref{final ein der} for organism is
\begin{equation}\label{c_final ein der}
    \tau_{u} \frac{\partial u}{\partial t}=\Delta_{e,u}\frac{\partial u}{\partial x}+\frac{1}{2} \sigma^2_{u} \frac{\partial^2 u}{\partial x^2}+\int_{t}^{t+\tau_{u}} f_{u}(x,\xi)d\xi.
\end{equation}

\begin{hypo}\label{prob func on s}
The chemotactic response of the organism $u(x,t)$ in the medium is influenced by the presence of chemical substrate $v(x,t)$. Therefore we hypothesize that probability density function $\varphi$ not only depends on the length of free jumps $\Delta$ but also on the spatial gradient of concentration of substrate $v(x,t)$ present in the medium. To be more specific, we assume that chemotactic response, which causes the event of movement of the organism towards food (or any attractor), is proportional to relative changes of $v$ in space with respect to the amount of food. Then by Definitions~\ref{axiom:Delta-e} and \ref{sigma-b}, $\Delta_{e,u}$ and $\sigma_{u}$ also depend on $v(x,t)$. Mathematically, the dynamics of directed movement characterised by expected value of free jump  $\Delta_{e,u}$ is 
\begin{equation}\label{Delab-def-A}
 \Delta_{e,u}(v) = -\beta \frac{1}{v} \frac{\partial v}{\partial x}= -\beta \frac{\partial \ln v}{\partial x}
 \end{equation}
with $\beta$ being a positive chemotactic coefficient and having dimension $ [L^2]$, which can be interpreted as a chemotactic factor for classification of the living organism.  
\end{hypo}

\begin{assumption}
Although in real life, standard deviation is a composite parameter depending on $v, \ \nabla  v, \ u, \nabla u, \ x, \ t, etc$,  in this article, we consider the dynamics of processes with constant standard deviation. 
Namely
\begin{equation}\label{sigma-def}
 \sigma^{2}_{u} (v)= \mu,   
\end{equation}
where $\mu$ is the motility parameter or diffusion coefficient of the organism with dimension $[L^2]$. 
We also assume that chemotactic factor $\beta$ is constant. Both $\mu$ and $\beta$  can be obtained from analyses of the dynamics of process, using image processing. 
\end{assumption}

\begin{hypo}\label{croud-movement}
$f_u$ is the number of organism per unit volume that form the crowd through quorum sensing in the domain containing chemical substrate. Therefore, $f_u$ depends on both $u$ and $v$. Let $\gamma_{0}$ is the rate of certain threshold concentration of organism (quorum) to be present in the crowd formation to trigger the event of accumulation of organism. We will call the coefficient $\gamma_0$ quorum rate.
Therefore, we define the rate of the movement of the organism to be proportional to the gradient (spatial changes) of expected free jump of substrate $v(x,t)$ per cell,
 \begin{equation}\label{Fb-def}
 \int_{t}^{t+\tau_{u}} f_u(x,\xi)d\xi \approx \tau_u u(x,t) F_u (u,v) = \tau_{u} u \gamma_{0}\frac{\partial \Delta_{e,u} (v)
 }{\partial x} =-u \tau_{u} \gamma \frac{\partial^2 \ln v
 }{\partial x^2}.
\end{equation}
Here $\gamma = \beta \gamma_{0}$ is the crowd effect stimulation coefficient, a positive constant with dimension $[ \frac{L^2}{T}]$. 
\end{hypo}

Therefore, under above assumptions we get,

\begin{equation}\label{b_sys}
         \tau_u \frac{\partial u}{\partial t}  = -\beta \frac{\partial \ln v}{\partial x}\frac{\partial u}{\partial x}+ \frac{\mu}{2}\frac{\partial^2 u}{\partial x^2} -\tau_{u}\gamma u\frac{\partial^2 \ln v}{\partial x^2}.
\end{equation}
 The first term on the right hand side of Eq.~\eqref{b_sys} is the chemotactic response of the organism, i.e., part of the flux of organism due to chemotaxis is proportional to the chemical gradient. The second term is the change in the density of organism due to random motion.  And the last term on the right represents the crowd formed by the complex interactions between organism- organism and organism-substrate.\par

And the concentration $v(x,t)$ of chemical substrate can be given by the equation,
\begin{equation}\label{der_s_sys}
    \tau_v \frac{\partial v}{\partial t} = \frac{\partial v}{\partial x} \Delta_{e,v} +\frac{1}{2}\sigma^2_v\frac{\partial^2 v}{\partial x^2}+ \int_{t}^{t+\tau_{v}} f_v(x,\xi) d\xi
\end{equation}

\begin{assumption} Food (chemical substrate) is considered to be immovable, so no chemical interaction between particles of substrates is possible under our assumption. Hence,
$$\Delta_{e,v}=0,$$
and
$$\sigma^{2}_{v} = D,$$ 
with $D$ being the diffusion constant of chemical substrate. 
\end{assumption} 

\begin{assumption} $f_v$ is defined to be the consumption by substrate cells, 
\begin{equation*}\label{cons-of-s}
\int_{t}^{t+\tau_{v}} f_v(x,\xi) d\xi= \tau_{v} F_v(u,v)= - \tau_{v}k(v)u,
\end{equation*}
where $k(v)$ is the rate of consumption of the substrate with dimension $[\frac{1}{T}]$.
\end{assumption}

Under assumptions, Eq.~\eqref{der_s_sys} can be written as
\begin{equation*}\label{s_sys}
    \tau_{v} \frac{\partial v}{\partial t}  = -\tau_{v} k(v)u+ D \frac{\partial^2 v}{\partial x^2}.
\end{equation*}

\begin{assumption}
We will assume $D=0$ and in the presence of abundance of substrate, the rate of the consumption of the food $k(v)$ does-not depend on the concentration of the food.Therefore,
\begin{equation}\label{k-const}
    k(v)=k=constant.
\end{equation}
\end{assumption}

Finally, the reduced system of equations are
\begin{subequations}\label{mod2}
 \begin{align}
     L_1 u &=\tau \frac{\partial u}{\partial t} +\beta \frac{\partial \ln v}{\partial x}\frac{\partial u}{\partial x}- \frac{\mu}{2}\frac{\partial^2 u}{\partial x^2} +\tau \gamma u\frac{\partial^2 \ln v}{\partial x^2}= 0 \label{red b_sys}\\
     L_2 v &=\frac{\partial v}{\partial t} +k u\label{red s_sys}= 0
 \end{align}
 \end{subequations}
 
 In above system, for simplicity,  $\tau=\tau_u$.
 We will consider so called family of solutions with no initial data that exhibits traveling wave phenomena. Namely, we will construct baseline solution which, at any given time $t$, will have the shifted invariant.
The only constraint which will be imposed will be at $\pm\infty.$

\begin{remark}
$\gamma_0 \propto \frac{1}{\tau}$. Since, concentration of organism $u \propto \frac{1}{\tau}$, then $u \propto \gamma_0$. Which can be interpreted physically as in the presence of less concentration of organism the quorum rate can afford to be smaller to stimulate the crowd effect phenomena. 
\end{remark}

\begin{definition}\label{travel_band} (Traveling Band)
Consider x vary from $+\infty$ to $-\infty$. A system exhibit traveling band if the solutions are in the following form
\begin{equation}\label{chng var}
    u(x,t)=u(\zeta) \text{,    } v(x,t)=v(\zeta) \text{,     } \zeta=x-c t
\end{equation}
where $c>0$ is the constant band speed.
\end{definition}

The following theorems can be proved to show that the system described above exhibits traveling wave phenomena.

\subsection{ Model without crowd effect for unlimited substrates
}\label{unlimited_alpha=0}
If $\beta - \gamma \tau = 0$, i.e., $\gamma_0 =\frac{1}{\tau}$, then our model gives the following form similar to the classic Keller-Segel model \cite{KELLER1971225} \& \cite{KELLER1971235}:
\begin{subequations}\label{mod1}
 \begin{align} 
     L_{1,0} u &=\tau \frac{\partial u}{\partial t} +\beta \frac{\partial}{\partial x}\left(u \frac{\partial \ln v}{\partial x}\right)- \frac{\mu}{2}\frac{\partial^2 u}{\partial x^2} = 0\label{red b_sys_eq}\\
     L_{2,0} v &=\frac{\partial v}{\partial t} +k u\label{red s_sys_eq}= 0
 \end{align}
\end{subequations}

\begin{theorem}\label{unlimited-food}
With the solution in the form of Eq.~\eqref{chng var}, the system \eqref{mod1} exhibit traveling band form in solution.
\end{theorem}

\begin{proof}
With the convention $\frac{d}{d\zeta}='$ and Eq.~\eqref{chng var}, the Eqs.~\eqref{red b_sys_eq} and \eqref{red s_sys_eq} are reduced to

\begin{align}
    L_{1,0} u &= \tau c u^{'} - \beta \big(u v^{-1} v^{'} \big)^{'}+ \frac{\mu}{2} u^{''} \label{ch red b_sys}=0,\\
    L_{2,0} v &= c v^{'}  - k u\label{ch red s_sys}=0.
\end{align}

And the appropriate conditions at $\pm\infty$ are
\begin{equation}\label{bd_cond}
    u \xrightarrow{} 0 \text{,    } u^{'} \xrightarrow{} 0 \text{,     } v \xrightarrow{} v_\infty \text{,      as   } \zeta \xrightarrow{} \infty
\end{equation}
where $v_{\infty}$ is positive constant.

First integrate Eq.~\eqref{ch red b_sys} once and obtain,
\begin{equation}\label{first_int_b}
    \tau c u - \beta u v^{-1} v^{'} + \frac{\mu}{2} u^{'} + constant=0.
\end{equation}

By conditions \eqref{bd_cond}, constant is 0. Then dividing the Eq.~\eqref{first_int_b} by $u$, we have

\begin{equation*}\label{(ln u)'-equation}
   (\tau c\zeta +\frac{\mu}{2}\ln u)' = (\beta \ln v)^{'}.
\end{equation*}

Integrating gives,
\begin{equation}\label{soln b_sys_old}
    u= C_{1} v^{\frac{2 \beta}{\mu}} e^{-\frac{2\tau c \zeta}{\mu}}.
\end{equation}
Here $C_{1}$ is the constant of integration and is positive.

Substituting the expression of $u$ into Eq.~\eqref{ch red s_sys} and integrating with conditions \eqref{bd_cond}, we get

\begin{equation}\label{soln s_sys}
    v=\bigg[\frac{1}{2}C_{1} k c^{-2}\tau^{-1} \mu (\frac{2 \beta}{\mu}-1) e^{-\frac{2\tau c \zeta}{\mu}}+v_{\infty}^{-\frac{2 \beta}{\mu}+1}\bigg]^{-\frac{1}{\frac{2 \beta}{\mu}-1}}.
\end{equation}

If we consider the constrain
\begin{equation}\label{mu_res}
   d = \frac{2 \beta}{\mu}>1  
\end{equation}

The solution \eqref{soln b_sys_old} satisfies,
\begin{equation*}
    \lim_{\zeta \xrightarrow{} \infty}u = 0 \text{     and     } \lim_{\zeta \xrightarrow{} -\infty}u = 0.
\end{equation*}

Also, due to the assumption \eqref{mu_res}, Eq.~\eqref{soln s_sys} exhibits the following behavior, 
\begin{equation*}
    \lim_{\zeta \xrightarrow{} \infty}v = v_{\infty} \text{     and     } \lim_{\zeta \xrightarrow{} -\infty}v = 0.
\end{equation*}

\end{proof}

\begin{proposition}
Consider the setting $\frac{1}{2}C_{1} k c^{-2}\tau^{-1}\mu (d-1)=v_{\infty}^{-\frac{2\beta}{\mu}+1}$ with the intention of achieving simplest expression for $v(\zeta)$ in Eq.~\eqref{soln s_sys}. Then function of $u$ is not monotone and maximum is achieved with the value,
\begin{equation*}
    u_{max} = 2 c^{2} \tau k^{-1} \mu^{-1} v_{\infty} d^{-(\frac{d}{d-1})}
\end{equation*}

at \begin{equation*}
    \zeta=\frac{\mu}{2\tau c}\ln(\frac{1}{d -1}).
\end{equation*}
Where the function $v$ is monotonically increasing from $zero$ to constant $v_\infty$.
Note that, $v_\infty$ is the certain threshold concentration of food that initiates the consumption of food by any living organism.
\end{proposition}

\begin{proof}
With Eq.~\eqref{soln s_sys}, we get 
\begin{equation}
    v = v_{\infty} \big(1 + e^{-\frac{2 \tau c}{\mu}\zeta}\big)^{-\frac{1}{d-1}}\label{cor_s_sys_mod1}
\end{equation}

and the corresponding expression for $u$ is
\begin{align}
    u 
    = &\frac{2 c^{2} \tau k^{-1} \mu^{-1}}{d -1} v_{\infty} \big(e^{-\frac{2\tau c}{\mu}\zeta}+1\big)^{-\frac{d}{d - 1}}e^{-\frac{2\tau c}{\mu}\zeta}.\label{cor_b_sys_mod1}
\end{align}

Differentiating Eq.~\eqref{cor_b_sys_mod1} with respect to $\zeta$, we get
\begin{align*}
    u^{'} &= -\frac{2\tau c}{\mu} u  \Big(1-\frac{d}{d -1}  \big(1+e^{\frac{2\tau c}{\mu}\zeta}\big)^{-1}\Big) .
\end{align*}

Then, $u_{max}$ occurs at $\zeta=\frac{\mu}{2\tau c}\ln(\frac{1}{d-1})$. And, so
\begin{align*}
    u_{max} &= 2 c^{2} \tau k^{-1} \mu^{-1} v_{\infty} d^{-(\frac{d}{d-1})}.
\end{align*}
\end{proof}


\subsection{ Model with crowd effect for unlimited substrates
}
\label{lim_alpha_neq_0}
In this section, we consider biological system when quorum rate is $\gamma_0\neq \frac1\tau$. We will also assume that pattern for the  food itself is scaled by factor $\exp{(\lambda t)}$ compare to the base-line case, and is subject to scale the traveling-band pattern such in previous section.\par

If $\alpha = \gamma \tau- \beta \neq 0$, then Eqs.~\eqref{red b_sys} and \eqref{red s_sys} are reduced to:
\begin{align}
    L_{1} u &=\tau \frac{\partial u}{\partial t} +\beta \frac{\partial }{\partial x}\bigg(\frac{\partial \ln v}{\partial x} u\bigg)- \frac{\mu}{2}\frac{\partial^2 u}{\partial x^2} +(\gamma \tau- \beta) u\frac{\partial^2 \ln v}{\partial x^2}=L_{1}u+\alpha u \frac{\partial^2 \ln v}{\partial x^2}=0,\label{mod1:not_zero_b}\\
    L_{2} v &= \frac{\partial v}{\partial t} + ku=0. \label{mod1:not_zero_s}
\end{align}

\begin{theorem}\label{thm_mod2}
Assume that
\begin{equation}\label{L0u0=0}
    L_{1,0}u=0,\text{   and   } L_{2,0}v = 0
\end{equation}
We compute
\begin{equation}\label{ln_constr}
    \max \bigg|\frac{\partial^2 \ln v}{\partial x^2}\bigg|=B = \frac{1}{d-1} \frac{ \tau^{2} c^{2}}{\mu^{2}}>0.
\end{equation}
Here $B$ is a positive constant with dimension $[L^{-2}]$. Then there exists constants $\lambda_{\pm}$ such that  
\begin{equation}\label{u+}
u_{\pm}(x,t)=e^{\lambda_{\pm} t}u(x,t)   \text{   and   } v_{\pm} = e^{\lambda_{\pm} t} v_(x,t)  
\end{equation}  
solves the partial differential inequality
\begin{equation*}
L_{\alpha,1}u_{+}\geq 0 \text{ ,and  }  L_{\alpha,1}u_{-}\leq 0. 
\text{   and   }   L_{\alpha,2}v_{+}\geq 0 \text{ ,  } L_{\alpha,2}v_{-}\leq 0 
\end{equation*}
\end{theorem}

\begin{proof}
Note that, $\ln v_{\pm} = \gamma \tau + \ln v$.\par

Considering the mapping \eqref{u+} and Eq.~\eqref{L0u0=0}, Eq.~\eqref{mod1:not_zero_b} \& \eqref{mod1:not_zero_s} becomes
\begin{align*}
    L_{1} u_{\pm} 
    &= \big(\lambda_{\pm} \tau +\alpha \frac{\partial^2 \ln v}{\partial x^2} \big) e^{\lambda_{\pm} t} u(x,t)  \label{u-pm-gen}\\
    &= \bigg(\tau \lambda_{\pm} -\alpha \frac{1}{d-1}\frac{4 \tau^2 c^2}{\mu^2} e^{\frac{2 \tau c}{\mu}\zeta}(1+e^{\frac{2 \tau c}{\mu}\zeta})^{-2} \bigg) e^{\lambda_{\pm} t} u(x,t), \\
    L_{\alpha} v_{\pm} & = \lambda e^{\lambda_{\pm} t}  v(x,t) 
\end{align*}
Using the computation \eqref{ln_constr} and assuming the existence of  
\begin{equation}\label{lambda-}
     \lambda_{-}= -\frac{\alpha B}{\tau}
\end{equation}
with $\alpha>0$, it follows
\begin{align*}
    L_{\alpha} u_{-} & \leq 0\\
    L_{\alpha} v_{-} & \leq 0
\end{align*}

Similarly,
\begin{align}\label{lambda+}
    \lambda_{+}= \frac{\alpha B}{\tau}
\end{align}
with $\alpha>0$ gives us
\begin{align*}
    L_{\alpha} u_{+} & \geq 0\\
    L_{\alpha} v_{+} & \geq 0
\end{align*}

Therefore, due to maximum principle, if $u$ and $v$ are the analytical solutions of the system \eqref{mod1} then there exists $\lambda_{\pm}$ given by Eqs.~\eqref{lambda-} and \eqref{lambda+} such that
\begin{align*}
    e^{\lambda_{-} t} u &\leq u \leq e^{\lambda_{+} t} u\\
    e^{\lambda_{-} t} v &\leq v \leq e^{\lambda_{+} t} v
\end{align*}
\end{proof}

\begin{remark}
For the system \eqref{mod2}, the analytical solution is not achievable for unlimited source of substrates. But, we can estimate lower and upper estimates, $u_{-}$ and $u_{+}$ respectively, for the solution $u$ that exhibits traveling band. Also, if $\alpha<0$ in Thm.~\ref{thm_mod2} then for the same values in Eqs.~\eqref{lambda-} and \eqref{lambda+}, the lower and upper estimates follow as $u_{+} \leq u \leq u_{-}$.
\end{remark}

\begin{remark}
Dimension of $\alpha$ is $[L^2]$ and therefore dimension of $\lambda$ is $ [\frac{1}{T}]$. In that sense, $u_{\pm}$ and $v_{\pm}$ are dimensionless.
\end{remark}


\section{Derivation of Chemotactic system when the availability of substrate is limited:}
\label{limited}
if the  unavailability of the source of food plays a role in depletion of concentration of substrate then $k(v) \propto v(x,t)$. Therefore the chemotactic model for unlimited substrate is reduced to
\begin{subequations}\label{mod3}
\begin{align}
     L_{\alpha,3}u & = \tau \frac{\partial u}{\partial t} +\beta \frac{\partial \ln v}{\partial x}\frac{\partial u}{\partial x} - \frac{\mu}{2}\frac{\partial^2 u}{\partial x^2} +\gamma \tau u\frac{\partial^2 \ln v}{\partial x^2}=0\label{ks1_b_sys}\\
     L_{\alpha,4}v & = \frac{\partial v}{\partial t} + k u v\label{ks1_s_sys}=0
\end{align}
\end{subequations}

\begin{theorem}
If the solution is in the form of Eq.~\eqref{chng var} then the system \eqref{mod3} exhibit traveling band phenomena. 
\end{theorem}

\begin{proof}
With Eq.~\eqref{chng var}, the system of equations \eqref{ks1_b_sys} and \eqref{ks1_s_sys} reduce to,

\begin{align}
  L_{\alpha,3}u &= \tau c u^{'} - \beta (\ln v)^{'} u^{'} +\frac{\mu}{2} u^{''} - \tau \gamma u (\ln v)^{''}=0 \label{ks_red_b_sys} \\
  L_{\alpha,4}v & = c v^{'} -k u v=0\label{ks_red_s_sys}  
\end{align}

Note that, Eq.~\eqref{ks_red_s_sys} gives, 
\begin{align*}
    (\ln v)^{''}=\frac{k}{c}u^{'}.
\end{align*} 

Therefore from Eqs.~\eqref{ks_red_b_sys} and \eqref{ks_red_s_sys}, it follows that
\begin{equation}\label{main-ODE-2-model}
    u^{'} - C_{3} (u^2)^{'}+\frac{\mu}{2 \tau c} u^{''}=0
\end{equation}

Here
\begin{equation*}\label{alpha}
    C_3 =\frac{1}{2}\frac{k}{\tau c^2}(\beta+ \gamma \tau).
\end{equation*}
with dimension $[L]$.

Then integration of Eq.~\eqref{main-ODE-2-model} gives,
\begin{equation*}\label{1-st-order-ode-2}
    u - C_{3} u^2+\frac{\mu}{2 \tau c} u^{'}= const.
\end{equation*}

By the condition \eqref{bd_cond}, $const.=0.$  

Then from above, it follows,
\begin{align*}\label{1-st-order-ode-2}
   \frac{u^{'}}{u(C_{4}-u)}&=-\frac{2 \tau c C_{3}}{\mu}
\end{align*}
with
\begin{equation*}\label{K-def}
C_{4}=\frac{1}{C_{3}}
\end{equation*}
which is dimensionless.

Partial decomposition gives, 
\begin{align*}\label{1-st-order-ode-2}
    \frac{u^{'}}{C_4 -u}+\frac{u^{'}}{u} & = -\frac{2 \tau c }{\mu} 
\end{align*}

And integration gives,
\begin{equation}\label{lim_neq_u}
    u= \frac{2 \tau c^2}{k(\beta + \gamma \tau)} \big(1+C_{5}e^{\frac{2 \tau c}{\mu}\zeta}\big)^{-1}.
\end{equation}
$C_{5}$ is the integrating constant.

Substituting Eq.~\eqref{lim_neq_u} into Eq.~\eqref{ks_red_s_sys} and integrating, we get
\begin{align}\label{lim_neq_v}
 v & = C_6 \big(e^{-\frac{2 \tau c}{\mu}\zeta}+C_5)^{-\frac{\mu} {\beta+ \gamma \tau}}
\end{align}
where $C_6$ is the integrating constant.

With condition \eqref{bd_cond}, it follows
\begin{align}
 v = v_{\infty} \left(1+C_{7} e^{-\frac{2 \tau c}{\mu}\zeta} \right)^{-\frac{\mu} {\beta+ \gamma \tau}}\label{lim_neq_v_cond}
\end{align}
with $C_{5}^{-1}=C_{7}$.

Then Eq.~\eqref{lim_neq_u} and \eqref{lim_neq_v_cond} satisfies
\begin{align*}
    \lim_{\zeta \xrightarrow{} \infty}u = 0 \text{     and     } &\lim_{\zeta \xrightarrow{} -\infty}u = \frac{2 \tau c^2}{k (\beta + \gamma \tau)} = \frac{2 \tau c^2 k^{-1} \beta^{-1}}{ 1 + \gamma_0 \tau},\\
    \lim_{\zeta \xrightarrow{} \infty}v = v_{\infty} \text{     and     } &\lim_{\zeta \xrightarrow{} -\infty}v = 0
\end{align*}

\end{proof}

\begin{remark}
Variable $\zeta =x-ct$ is of the form of traveling waves, and it converges to $-\infty$ if $x\to -\infty$ for fixed time $t$ or as $t\to \infty$ for fixed $x.$ In another more lenient word, ``when  space meet the time". Here, the concentration of organism $u(x, t)$ tends to converge to the constant that is inversely proportional to the quorum rate $\gamma_0$ as $\zeta \to -\infty$. 
\end{remark}

The next theorem explains the traveling band features for the model without crowd effect in the environment containing limited amount of substrates. 
\begin{theorem}
If $\gamma_0 = \frac{1}{\tau}$, our model becomes
\begin{subequations}\label{mod4}
 \begin{align} 
     L_3 u &=\tau \frac{\partial u}{\partial t} +\beta \frac{\partial}{\partial x}\left(u \frac{\partial \ln v}{\partial x}\right)- \frac{\mu}{2}\frac{\partial^2 u}{\partial x^2} = 0\label{ks1_ b_sys_eq}\\
     L_4 v &=\frac{\partial v}{\partial t} +k u v\label{ks1_ s_sys_eq}= 0
 \end{align}
\end{subequations}
then the above system exhibits traveling band phenomena.
\end{theorem} 

\begin{proof}
With Eq.~\eqref{chng var}, the system of equation reduces to,
\begin{align}
  L_{3} & = \tau c u^{'} - \beta (u (\ln v)^{'})^{'}  + \frac{\mu}{2} u^{''} = 0\label{ks1_red_b_sys} \\
  L_{4} & = c v^{'}  -  k u v = 0\label{ks1_red_s_sys}  
\end{align}

From Eq.~\eqref{ks1_red_s_sys}, we find
\begin{align}\label{s_ln_sys}
    (\ln v)^{'} = \frac{k}{c}u
\end{align}

Then Eq.~\eqref{ks1_red_b_sys} becomes,
\begin{align}
    u^{'} - \frac{\beta k}{\tau c^2} (u^2)^{'}+\frac{\mu}{2 \tau c} u^{''} &= 0 \label{ks1_int1_b_sys}
\end{align}

Integrating Eq.~\eqref{ks1_int1_b_sys} and applying \eqref{bd_cond}, we get
\begin{align*}
    \frac{u^{'}}{u \big(1-\frac{\beta k}{\tau c^2}u\big)} &= -\frac{2 \tau c}{\mu}\label{ks1_pf_b_sys}
\end{align*}
With partial fraction decomposition and integrating, we get

\begin{equation}\label{ks1_b_soln}
    u =\frac{1}{\frac{\beta k}{\tau c^2}+C_8 e^{\frac{2 \tau c}{\mu}\zeta}}
\end{equation}
where $C_8$ is the integrating constant.

Upon setting $C_8 = \frac{\mu k}{2 \tau c^2}$, Eq.~\eqref{ks1_b_soln} becomes,
\begin{align}
    u &= 2 \tau c^2k^{-1} \mu^{-1}\bigg(d + e^{\frac{2 \tau c}{\mu}\zeta}\bigg)^{-1}\label{lim_eq_u}
\end{align}

Substituting Eq.~\eqref{lim_eq_u} into Eq.~\eqref{s_ln_sys} and integrating, we get
\begin{align*}
    v &= C_9 \bigg(\frac{\beta k}{\tau c^2}e^{-\frac{2 \tau c}{\mu}\zeta}+\frac{\mu k}{2 \tau c^2} \bigg)^{-\frac{\mu}{2 \beta}}
\end{align*}
with integrating constant $C_9$.

Applying ~\eqref{bd_cond}, we get
\begin{align}
    v &= v_{\infty} \bigg(d e^{-\frac{2 \tau c}{\mu}\zeta} +1\bigg)^{-\frac{1}{d}}\label{lim_eq_v}
\end{align}

The solutions \eqref{lim_eq_u} and \eqref{lim_eq_v} have the properties
\begin{align*}\label{prop_u_lim_eq}
   \lim_{\zeta \xrightarrow{} \infty}u = 0 \text{       and     } &\lim_{\zeta \xrightarrow{} -\infty}u = \frac{2 \tau c^2k^{-1} \mu^{-1}}{d} = \frac{\tau c^2}{k\beta},\\
    \lim_{\zeta \xrightarrow{} \infty}v = v_{\infty} \text{       and     } &\lim_{\zeta \xrightarrow{} -\infty}v = 0.
\end{align*}

\end{proof}

\begin{remark}
The concentration of organism converges to a constant which is inversely proportional to $d = \frac{2 \beta}{\mu}$ for long time. Recall that, $\beta$ is the rate of the relative flux of organism with respect to chemotactic response and $\mu$ is the motility coefficient. Here, chemotactic response is measured by the ratio of the gradient of chemotaxis with respect to its density. It is important to observe that in this case, unlike the previous one, both organisms and substrates are independent of the tail of crowd formed by the organism that chase after food, following leaders. More simplification of this expression tells us that the constant is proportional to time interval of collision and traveling speed but inversely proportional to the consumption rate and chemotactic coefficient.  
\end{remark}


\section{Discussion:}
\label{dis}
In this section, we provide analyses of obtained closed form solutions of all of the four models (\eqref{mod1}, \eqref{mod2}, \eqref{mod3} and \eqref{mod4}). We will discuss the traveling band phenomena in each cases. Since analytical solutions are not obtainable for the model \eqref{mod2} involving crowd effect in the presence of unlimited substrate, we will show the upper and lower estimates of the analytical solution featuring traveling band. 
All results are qualitative and we used the following listed parameter values that are adapted from published data \cite{Adleramino} and \cite{AdlerDahl} except $\tau$ and $\gamma_0$ for comparison. Our own data will be provided from future experiments based on fluorescence imaging. Review of corresponding technique is presented in article  \cite{akifpap1}. 

\begin{table}[H]
    \centering
    \begin{tabular}{c c c c}
    \hline
    Parameter & Description & Value & Units \\
    \hline
    $\tau$ & time interval of collision & 0.05-0.005 & hour\\
    $\mu$ & motility coefficient & 0.25 & cm$^2$/hour\\
    $c$ & band speed & 1.5 & cm/hour\\
    $\beta$ & chemotactic coefficient & 0.16-0.6 & cm$^2$/hour\\
    $d$ & $\frac{2 \beta}{\mu}$ & 0.3 - 5 & unit less\\
    $\gamma_0$ & quorum rate & 12-100 & 1/hour\\
    $C_5$ & Integrating constant & 1 & unit less\\
    \hline
    \end{tabular}
    \caption{Parameter values}
    \label{pv}
\end{table}

Graphs of the solutions of the concentration of organism for the system \eqref{mod1} for different values of $\tau$ are given in the Fig.~\ref{fig:unlim_eq_tau}. The size of the band is wider when $\tau$ gets smaller depicting the dependence of the size of the traveling band on the time collision $\tau$.   
\begin{figure}
    \centering
    \includegraphics[width=0.6\textwidth]{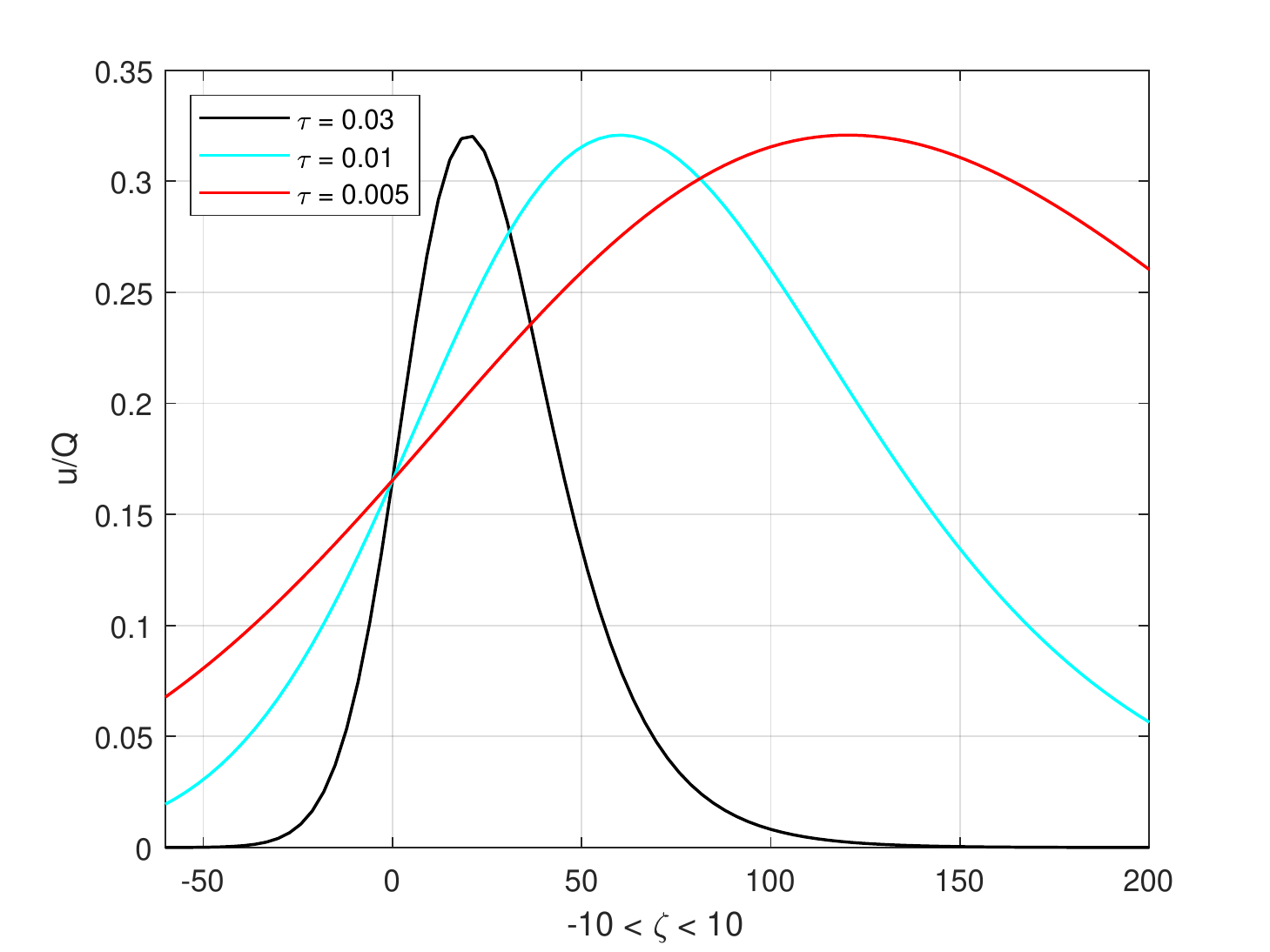}
    \caption{Concentration of organism $u(x,t)$ divided by $Q =  2 \tau c^{2} k^{-1}\mu^{-1} v_{\infty}$ of model \eqref{mod1} for different values of $\tau$ with $d = 1.3$ against $\zeta = c \mu^{-1} (x- c t)$.}
    \label{fig:unlim_eq_tau}
\end{figure}

Fig.~\ref{fig:ul_neq_bact_x_t} gives the upper estimate (magenta curve) and lower estimate (green curve) of the analytical solution of organism for the system \eqref{mod2} for different values of $d$. Graphs of first row is the estimate for a fixed $x$ value and of second row is for a fixed $t$ value. $\lambda$ has computed using the Eqs.~\eqref{lambda-} and \eqref{lambda+} for $\gamma_0 = 25$ and $\tau = 0.05$. Notice that, $\lambda$ is proportional to the quorum factor $\gamma_0$. And therefore $u_{+}$ and $u_{-}$ which are scaled by the factor of $exp (\lambda t)$ give better estimation when quorum rate is necessarily small for fixed $x$. 
\begin{figure}
    \centering
    \includegraphics[width=1\textwidth]{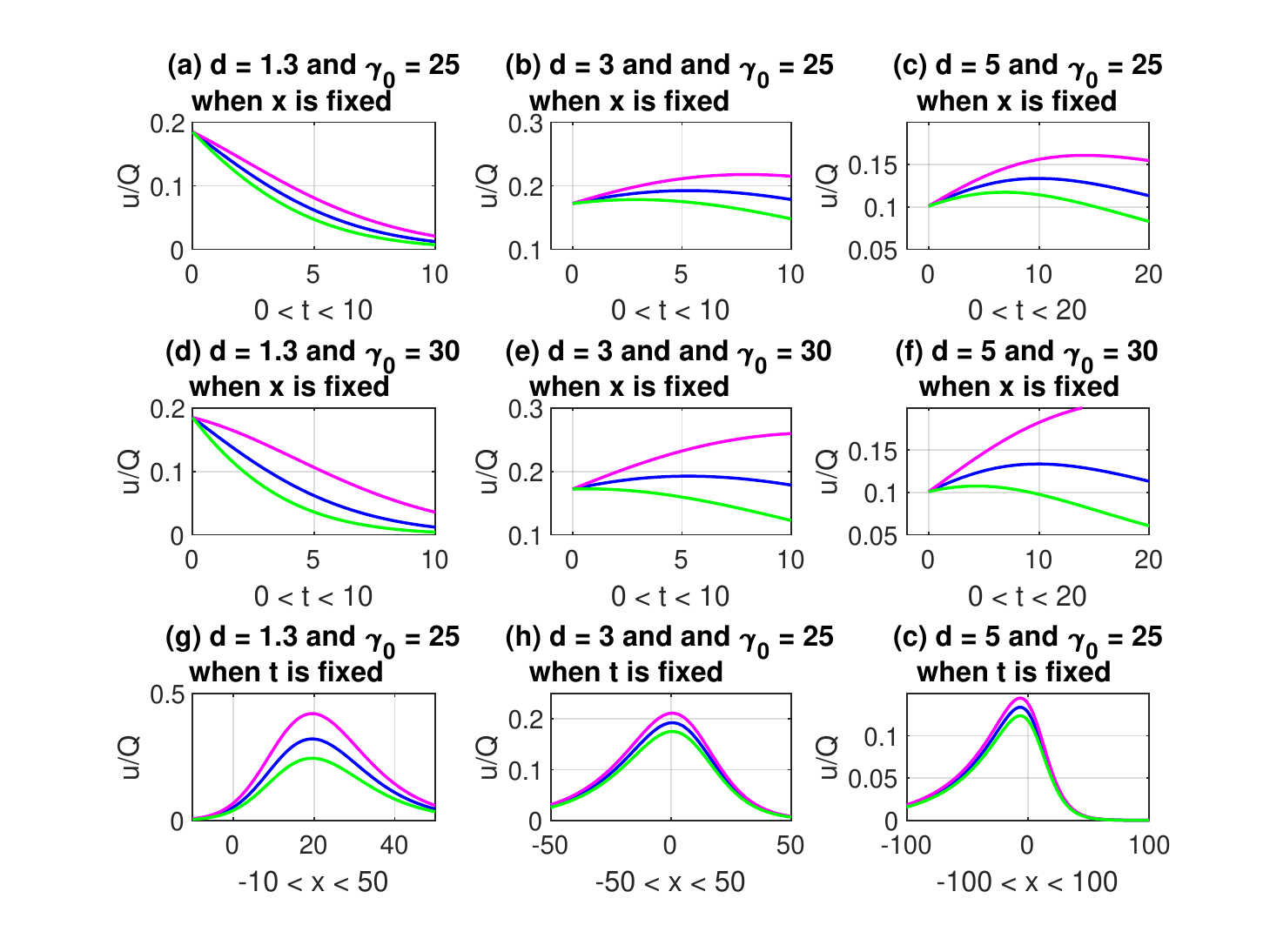}
    \caption{Concentration of organism $u(x,t)$ divided by $Q =  2 \tau c^{2} k^{-1}\mu^{-1} v_{\infty}$ of model \eqref{mod2} and its lower estimate (green curve) and upper estimate (magenta curve) for different values of $d$ and $\gamma_0$ when $t$ is fixed (first and second row) and $x$ is fixed (third row) in $\zeta = c \mu^{-1} (x- c t)$.}
    \label{fig:ul_neq_bact_x_t}
\end{figure}

The concentration of organism and of substrate of the model \eqref{mod3} in Figs.~\ref{fig:lim_neq} and \ref{fig:lim_neq_sub} for different values of $\gamma_0$. Concentration of $u$ converge to 0 as $\zeta \to \infty$. For large negative values of $\zeta$, the concentration $u$ converges to a constant that get reduced in size as $\gamma_0$ get larger.Therefore, in the presence of limited food, if we fix the location and look ahead for long time the concentration of organism will converge to a smaller constant if the quorum rate is bigger. Recall, quorum rate is the rate of the number of organism or cell in the crowd that inspires other to follow. And, in the case of substrate, the curve of $v$ gets flatter as $\gamma_0$ increases. 
\begin{figure}
    \centering
    \includegraphics[width=0.6\textwidth]{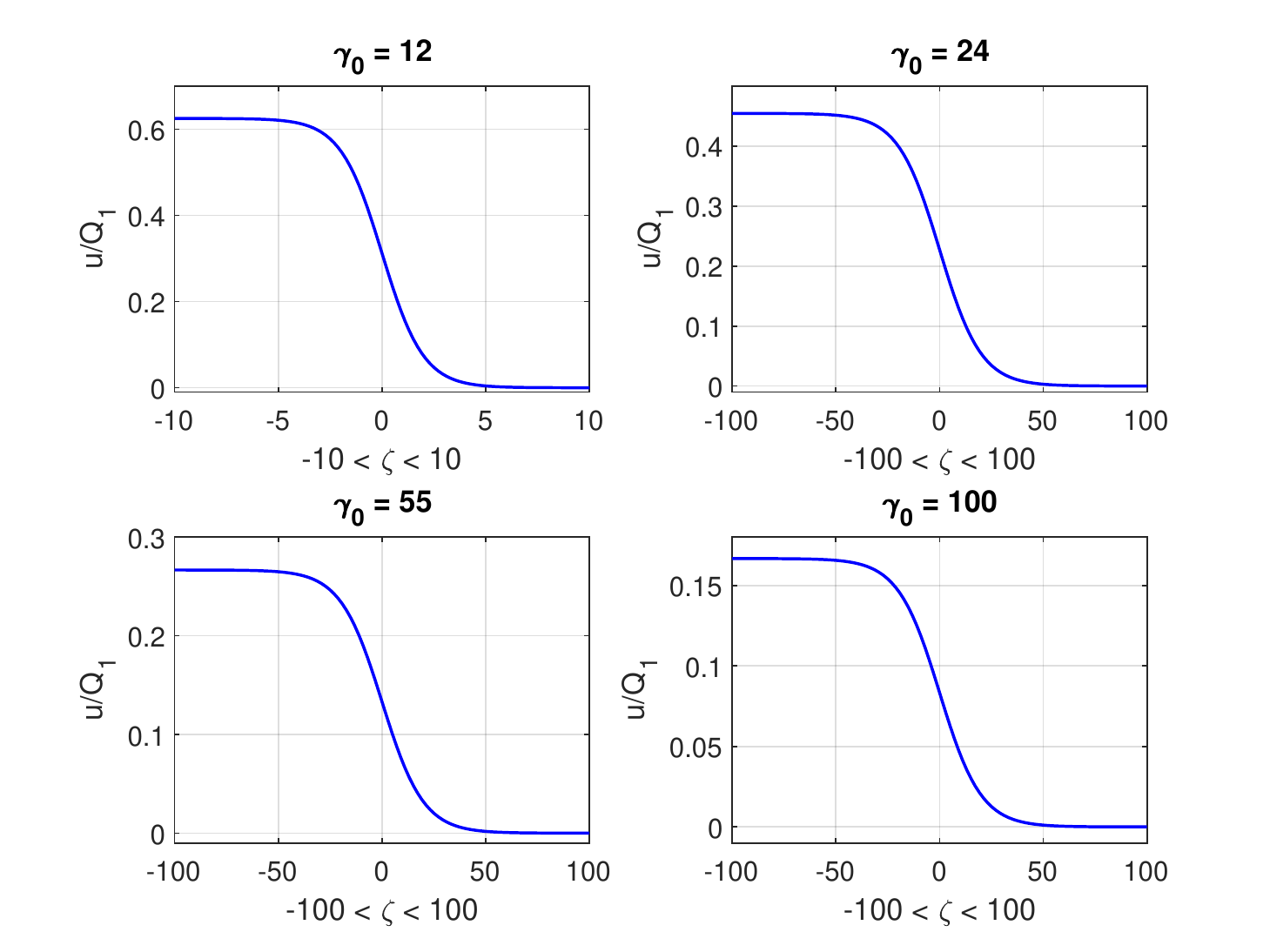}
    \caption{Concentration of organism $u(x,t)$ divided by $Q_{1} =  2 \tau c^{2} k^{-1}\beta^{-1}$ of model \eqref{mod3} for different values of $\gamma_0$ against $\zeta = c \mu^{-1} (x- c t)$.}
    \label{fig:lim_neq}
\end{figure}

\begin{figure}
    \centering
    \includegraphics[width=0.6\textwidth]{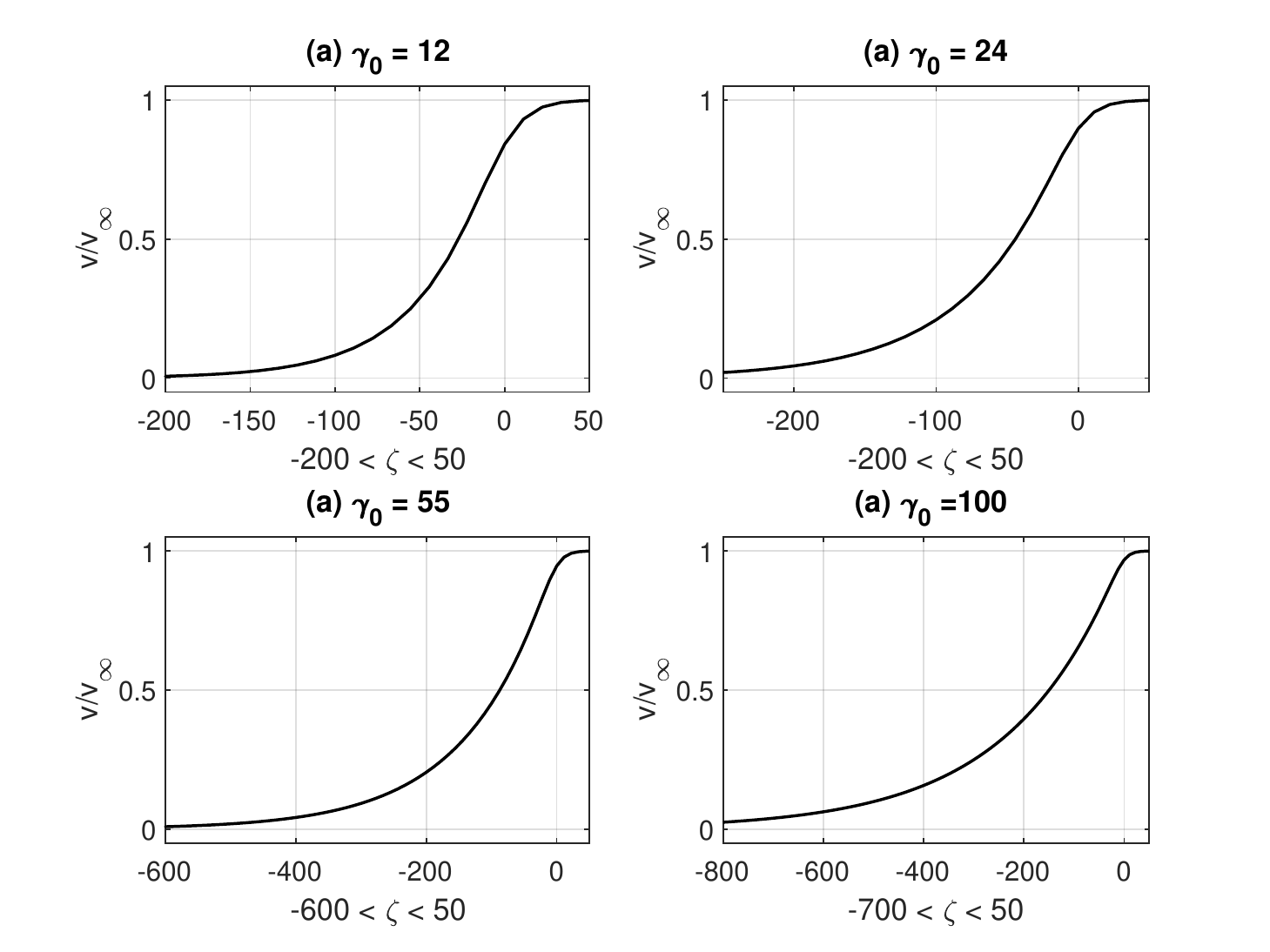}
    \caption{Same as in Fig. \ref{fig:lim_neq} for the concentration of substrate $v(x,t)$ divided by $v_{\infty}$ of model \eqref{mod3}.}
    \label{fig:lim_neq_sub}
\end{figure}

The graphs of Fig.~\ref{fig:lim_eq_bact} shows that concentration $u$ of model \eqref{mod4} converges to 0 as $\zeta \to -\infty$ and to a constant as $\zeta \to \infty$. This constant is inversely proportional to the value of $d$ in the absence of $\gamma_0$. Graphs of substrate in the Fig.~\ref{fig:lim_eq_subs} represent the substrate of model \eqref{mod4} for different $d$.  
\begin{figure}
    \centering
    \includegraphics[width=0.6\textwidth]{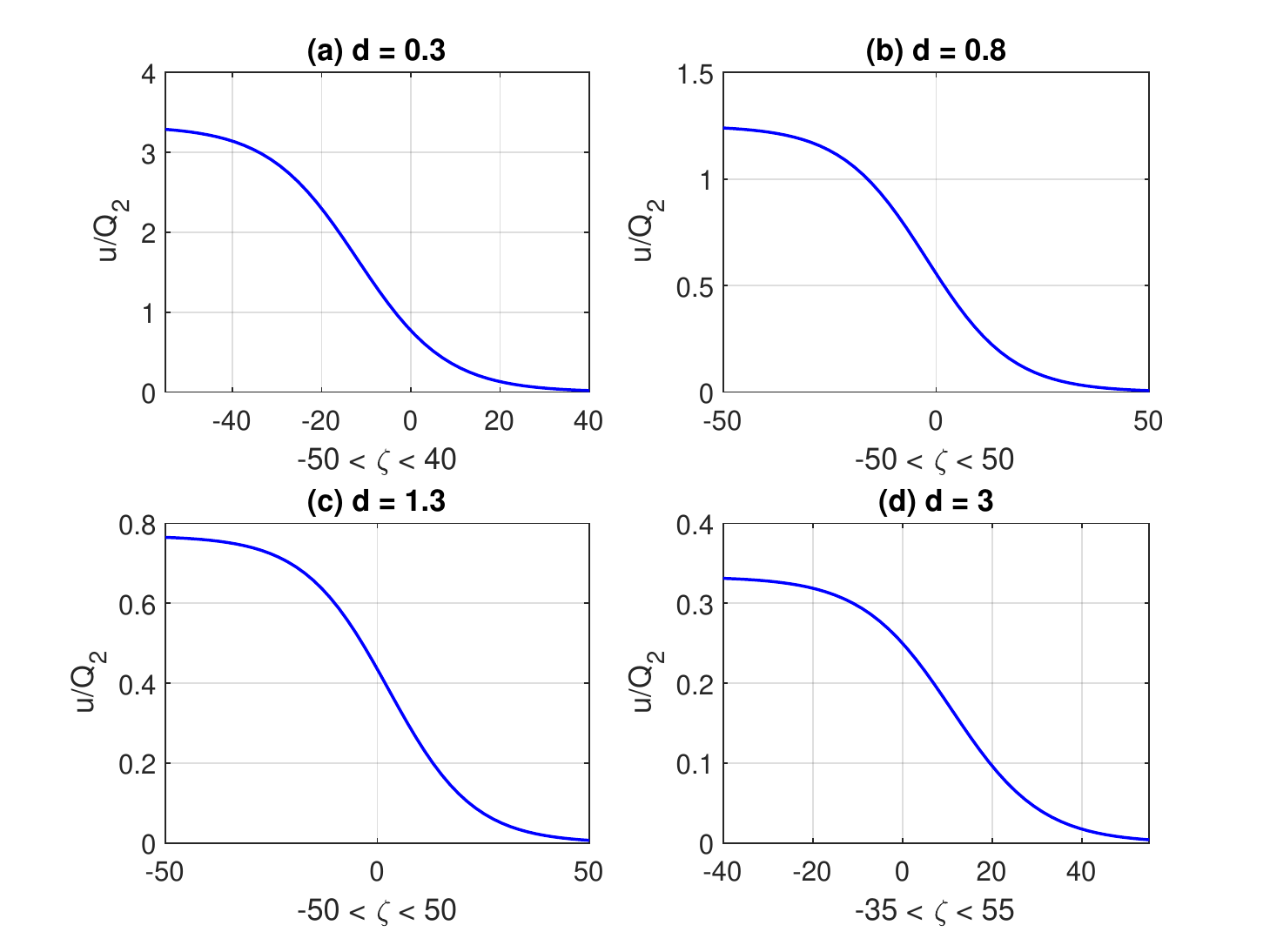}
    \caption{Concentration of organism $u(x,t)$ divided by $Q_{2} =  2 \tau c^{2} k^{-1}\mu^{-1}$ of model \eqref{mod4} for different values of $d$ against $\zeta = c \mu^{-1} (x- c t)$.}
    \label{fig:lim_eq_bact}
\end{figure}

\begin{figure}
    \centering
    \includegraphics[width=0.6\textwidth]{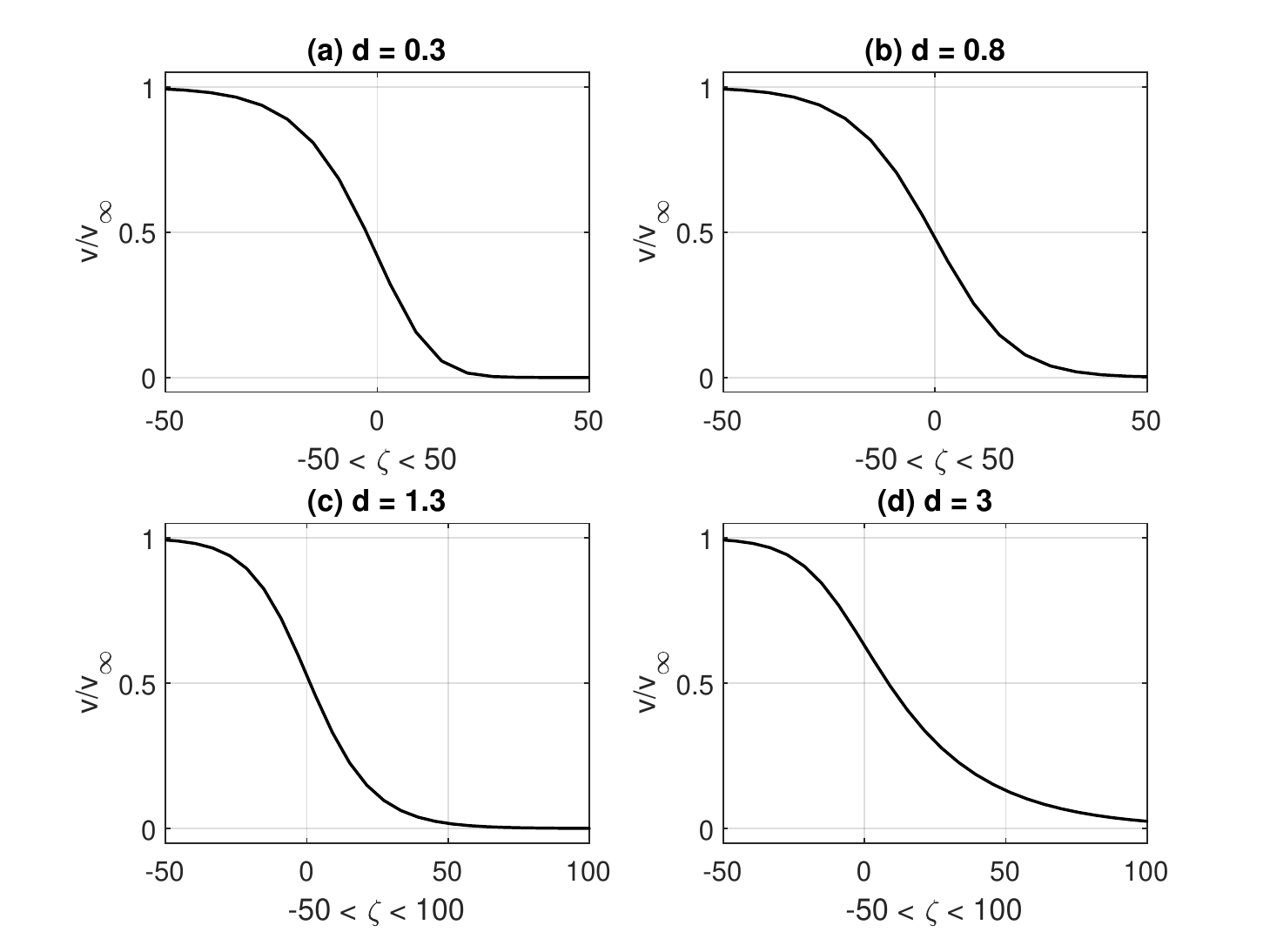}
    \caption{Same as in Fig. \ref{fig:lim_eq_bact} for the concentration of substrate $v(x,t)$ divided by $v_{\infty}$ of model \eqref{mod4}.}
    \label{fig:lim_eq_subs}
\end{figure}


\section{Conclusion:} In this study, we have shown a way of incorporation of Einstein's Brownian motion model in deducing chemotactic system. This involvement describes the dependence of the band size of the solution \eqref{cor_b_sys_mod1} for the model \eqref{mod1} on the time collision $\tau$. For unlimited food and without considering the crowd phenomena, we obtain analytical solution but its not the case for the model \eqref{mod2} that involves crowd effect. However, we were able to develop a mapping that enable us to obtain estimates for the traveling band of organism in the presence of unlimited food. These estimates were proved to be bound on the analytical solution with band form. \par

We have also described models when there is a limited presence of food in the environment. The model \eqref{mod3} with crowd effect shows that organism travel in a band form towards the presence of substrates. And the size of the organism for long time depends on the quorum rate $\gamma_0$. Therefore, if the rate of concentration of organism in a crowd that triggers the event of following crowd is larger then the size of the organism get smaller in the progression of time. The model \eqref{mod4} without crowd effect for limited food has also been explained to focus on the system where colonial formation doesn't occur. In this case, when $\gamma_0$ is no longer effective, organism still moves in traveling band form. But the size of the organism for long time depends inversely on the value $d$ which describes the ratio of chemotactic response and motility.


\nocite{*}
\bibliography{chemotaxis}

\providecommand{\latin}[1]{#1}
\makeatletter
\providecommand{\doi}
  {\begingroup\let\do\@makeother\dospecials
  \catcode`\{=1 \catcode`\}=2 \doi@aux}
\providecommand{\doi@aux}[1]{\endgroup\texttt{#1}}
\makeatother
\providecommand*\mcitethebibliography{\thebibliography}
\csname @ifundefined\endcsname{endmcitethebibliography}
  {\let\endmcitethebibliography\endthebibliography}{}
\begin{mcitethebibliography}{13}
\providecommand*\natexlab[1]{#1}
\providecommand*\mciteSetBstSublistMode[1]{}
\providecommand*\mciteSetBstMaxWidthForm[2]{}
\providecommand*\mciteBstWouldAddEndPuncttrue
  {\def\EndOfBibitem{\unskip.}}
\providecommand*\mciteBstWouldAddEndPunctfalse
  {\let\EndOfBibitem\relax}
\providecommand*\mciteSetBstMidEndSepPunct[3]{}
\providecommand*\mciteSetBstSublistLabelBeginEnd[3]{}
\providecommand*\EndOfBibitem{}
\mciteSetBstSublistMode{f}
\mciteSetBstMaxWidthForm{subitem}{(\alph{mcitesubitemcount})}
\mciteSetBstSublistLabelBeginEnd
  {\mcitemaxwidthsubitemform\space}
  {\relax}
  {\relax}

\bibitem[Einstein(1905)]{Einstein05}
Einstein,~A. \"{U}ber die von der molekularkinetischen Theorie der W\"{a}rme
  geforderte Bewegung von in ruhenden Fl\"{u}ssigkeiten suspendierten Teilchen.
  \emph{Ann. Phys. (Leipzig)} \textbf{1905}, \emph{322}, 549--560\relax
\mciteBstWouldAddEndPuncttrue
\mciteSetBstMidEndSepPunct{\mcitedefaultmidpunct}
{\mcitedefaultendpunct}{\mcitedefaultseppunct}\relax
\EndOfBibitem
\bibitem[Miller and Bassler(2001)Miller, and Bassler]{quorum01}
Miller,~M.~B.; Bassler,~B.~L. Quorum Sensing in Bacteria. \emph{Annual Review
  of Microbiology} \textbf{2001}, \emph{55}, 165--199, PMID: 11544353\relax
\mciteBstWouldAddEndPuncttrue
\mciteSetBstMidEndSepPunct{\mcitedefaultmidpunct}
{\mcitedefaultendpunct}{\mcitedefaultseppunct}\relax
\EndOfBibitem
\bibitem[M. \latin{et~al.}(2007)M., M., R., F., and A.]{GOBBETTI200734}
M.,~G.; M.,~D.~A.; R.,~R. D.~C.; F.,~M.; A.,~L. Cell to cell communication in
  food related bacteria. \emph{International Journal of Food Microbiology.}
  \textbf{2007}, \emph{120}, 34--45, 20th International ICFMH Symposium on FOOD
  MICRO 2006\relax
\mciteBstWouldAddEndPuncttrue
\mciteSetBstMidEndSepPunct{\mcitedefaultmidpunct}
{\mcitedefaultendpunct}{\mcitedefaultseppunct}\relax
\EndOfBibitem
\bibitem[Ward \latin{et~al.}(2008)Ward, Sumpter, Couzin, Hart, and
  Krause]{Ward6948}
Ward,~A. J.~W.; Sumpter,~D. J.~T.; Couzin,~I.~D.; Hart,~P. J.~B.; Krause,~J.
  Quorum decision-making facilitates information transfer in fish shoals.
  \emph{Proceedings of the National Academy of Sciences} \textbf{2008},
  \emph{105}, 6948--6953\relax
\mciteBstWouldAddEndPuncttrue
\mciteSetBstMidEndSepPunct{\mcitedefaultmidpunct}
{\mcitedefaultendpunct}{\mcitedefaultseppunct}\relax
\EndOfBibitem
\bibitem[Christov \latin{et~al.}(2020)Christov, Ibraguimov, and Islam]{rah20}
Christov,~I.~C.; Ibraguimov,~A.; Islam,~R. Long-time asymptotics of
  non-degenerate non-linear diffusion equations. \emph{J. Math. Phys.}
  \textbf{2020}, \emph{61}, 081505\relax
\mciteBstWouldAddEndPuncttrue
\mciteSetBstMidEndSepPunct{\mcitedefaultmidpunct}
{\mcitedefaultendpunct}{\mcitedefaultseppunct}\relax
\EndOfBibitem
\bibitem[Keller and Segel(1971)Keller, and Segel]{KELLER1971225}
Keller,~E.~F.; Segel,~L.~A. Model for chemotaxis. \emph{Journal of Theoretical
  Biology} \textbf{1971}, \emph{30}, 225--234\relax
\mciteBstWouldAddEndPuncttrue
\mciteSetBstMidEndSepPunct{\mcitedefaultmidpunct}
{\mcitedefaultendpunct}{\mcitedefaultseppunct}\relax
\EndOfBibitem
\bibitem[Keller and Segel(1971)Keller, and Segel]{KELLER1971235}
Keller,~E.~F.; Segel,~L.~A. Traveling bands of chemotactic bacteria: A
  theoretical analysis. \emph{Journal of Theoretical Biology} \textbf{1971},
  \emph{30}, 235--248\relax
\mciteBstWouldAddEndPuncttrue
\mciteSetBstMidEndSepPunct{\mcitedefaultmidpunct}
{\mcitedefaultendpunct}{\mcitedefaultseppunct}\relax
\EndOfBibitem
\bibitem[Adler(1966)]{Adleramino}
Adler, Effect of amino acids and oxygen on chemotaxis in Escherichia coli.
  \emph{Journal of bacteriology} \textbf{1966}, \emph{92}, 121--129\relax
\mciteBstWouldAddEndPuncttrue
\mciteSetBstMidEndSepPunct{\mcitedefaultmidpunct}
{\mcitedefaultendpunct}{\mcitedefaultseppunct}\relax
\EndOfBibitem
\bibitem[ADLER and DAHL(1967)ADLER, and DAHL]{AdlerDahl}
ADLER,~J.; DAHL,~M.~M. A Method for Measuring the Motility of Bacteria and for
  Comparing Random and Non-random Motility. \emph{Microbiology} \textbf{1967},
  \emph{46}, 161--173\relax
\mciteBstWouldAddEndPuncttrue
\mciteSetBstMidEndSepPunct{\mcitedefaultmidpunct}
{\mcitedefaultendpunct}{\mcitedefaultseppunct}\relax
\EndOfBibitem
\bibitem[Somaweera \latin{et~al.}(2016)Somaweera, Ibraguimov, and
  Pappas]{akifpap1}
Somaweera,~H.; Ibraguimov,~A.; Pappas,~D. A review of chemical gradient systems
  for cell analysis. \emph{Analytica Chimica Acta} \textbf{2016}, \emph{907},
  7--17\relax
\mciteBstWouldAddEndPuncttrue
\mciteSetBstMidEndSepPunct{\mcitedefaultmidpunct}
{\mcitedefaultendpunct}{\mcitedefaultseppunct}\relax
\EndOfBibitem
\bibitem[Adler(1966)]{Adler708}
Adler,~J. Chemotaxis in Bacteria. \emph{Science} \textbf{1966}, \emph{153},
  708--716\relax
\mciteBstWouldAddEndPuncttrue
\mciteSetBstMidEndSepPunct{\mcitedefaultmidpunct}
{\mcitedefaultendpunct}{\mcitedefaultseppunct}\relax
\EndOfBibitem
\bibitem[Adler(1975)]{Adlerchemo}
Adler, Chemotaxis in Bacteria. \emph{Annual Reviews Biochemistry}
  \textbf{1975}, \emph{44}, 341--356\relax
\mciteBstWouldAddEndPuncttrue
\mciteSetBstMidEndSepPunct{\mcitedefaultmidpunct}
{\mcitedefaultendpunct}{\mcitedefaultseppunct}\relax
\EndOfBibitem
\bibitem[Matsushita~M. \latin{et~al.}(2004)Matsushita~M., Kobayashi~N., and
  Yamazaki~Y.]{crowd}
Matsushita~M.,~H.~F.; Kobayashi~N.,~O.~T.; Yamazaki~Y.,~M.~T. Colony formation
  in bacteria: Experiments and modeling. \emph{Biofilms; Cambridge}
  \textbf{2004}, \emph{1}, 305--317\relax
\mciteBstWouldAddEndPuncttrue
\mciteSetBstMidEndSepPunct{\mcitedefaultmidpunct}
{\mcitedefaultendpunct}{\mcitedefaultseppunct}\relax
\EndOfBibitem
\end{mcitethebibliography}

\end{document}